\documentclass[12pt]{amsart}

\usepackage[utf8]{inputenc}
\usepackage[T1]{fontenc}
\usepackage{amsmath,amssymb,amsthm}
\usepackage{geometry}
\newtheorem{theorem}{Theorem}[section]
\newtheorem{proposition}{Proposition}[section]
\newtheorem{lemma}{Lemma}[section]
\newtheorem{corollary}{Corollary}[section]
\newtheorem{definition}{Definition}[section]
\newtheorem{remark}{Remark}[section]

\begin{document}

\title[A Rigidity Theorem for Flat Graph Hypersurfaces]{A Rigidity Theorem for Flat Graph Hypersurfaces of Variable-Returns-to-Scale Cobb--Douglas Production Functions}

\author[B.-Y. Chen]{Bang-Yen Chen}

\address{Department of Mathematics, Michigan State University,
                619 Red Cedar Road,
                East Lansing, 48824-1027 MI, USA}
\email{chenb@msu.edu}

\author[S. Deshmukh]{Sharief Deshmukh}

\address{Department of Mathematics, College of Science, King Saud
University,
P.O. Box 2455,
11451 Riyadh,
Saudi Arabia}
\email{shariefd@ksu.edu.sa}

\author[N.B. Turki]{Nasser Bin Turki}

\address{Department of Mathematics, College of Science, King Saud
University,
P.O. Box 2455,
11451 Riyadh,
Saudi Arabia}
\email{nassert@ksu.edu.sa}

\author[A.-D. V\^{\i}lcu]{Alina-Daniela V\^{\i}lcu}

\address{Department of Computer Science, Information Technology, Mathematics and Physics, Petroleum-Gas University of Ploie\c{s}ti,
                Bd. Bucure\c{s}ti 39,
                 100680 Ploie\c{s}ti,
                Romania}
\email{daniela.vilcu@upg-ploiesti.ro}

\author[G.-E. V\^{\i}lcu]{Gabriel-Eduard V\^{\i}lcu}

\address{"Gheorghe Mihoc-Caius Iacob" Institute of Mathematical Statistics
and Applied Mathematics of the Romanian Academy,
Calea 13 Septembrie 13,
050711 Bucharest,
Romania}

\address{Department of Mathematics and Informatics, National University of Science and Technology Politehnica Bucharest,
                313 Splaiul Independen\c{t}ei,
                 060042 Bucharest, Romania}
\email{gabriel.vilcu@upb.ro}

\begin{abstract}

We study graph hypersurfaces associated with
higher-dimensional ($n>2$) Cobb--Douglas production functions with
variable returns to scale. We prove a rigidity theorem showing that
every flat model in this class is necessarily homogeneous. The proof relies on an affine
dimension-reduction principle that transfers the higher-dimensional
flatness condition to suitable two-dimensional restrictions. We then
obtain a complete classification in the non-degree-one case: every
such flat model is a power of a positive linear form. For degree one,
we give a local structural description in terms of envelopes on the
rank-one set, with linearity on open rank-zero regions. These results
extend the two-input analysis recently developed in
\cite{Previous2D} to the higher-dimensional setting, where the
degree-one case exhibits a genuinely different behavior.
\end{abstract}

\maketitle

\section{Introduction}\label{S1}

The differential geometry of production functions has become an active
research area, revealing close connections between the geometry of graph
hypersurfaces and the analytical behavior of the associated production models.
Among the geometric conditions investigated in the literature, curvature conditions occupy a distinguished position because they often lead to rigidity phenomena: a strong geometric assumption may impose
severe restrictions on the admissible functional forms of production
technologies \cite{ACDV,AE2,AM,CS,CH2,CH4,CDV14,CVVMJM,DecuV,Du,FuLuo,FW,Luo,Luo2025,NTV,VV,VV3,VGE}.

The homogeneous classification of Chen and V\^\i lcu \cite{CV2013} provides a natural starting point for the geometric study of flat production hypersurfaces. Its non-degree-one branch leads to the familiar power-of-a-linear-form family. The degree-one converse, however, is dimension-dependent. In two variables, Euler's identity forces the Hessian rank to be at most one, so every $C^2$ degree-one homogeneous graph is flat. For $n>2$, the same identity supplies only the radial null direction and does not force flatness. The homogeneity assumption, however, represents a substantial
restriction. Several production models in the economics literature have
been described using the terminology of variable production elasticities
or variable exponents, allowing the corresponding parameters to vary with
the production technology or with the proportions of the production
factors. Early
examples include the transcendental production function and related
flexible specifications \cite{Halter}. Ulveling and Fletcher
\cite{Ulveling} introduced Cobb--Douglas models with variable production
elasticities, while de Janvry \cite{Janvry} incorporated similar models
into his framework of generalized power production functions and
interpreted them as Cobb--Douglas production functions with variable
returns to scale. More recently, Reyn\`es \cite{Rey} revisited this
perspective by considering Cobb--Douglas production functions with
input-dependent output elasticities.

In the setting of Cobb--Douglas production functions with variable
returns to scale, these developments naturally raise the following
question:
\emph{Does flatness of the graph hypersurface force homogeneity
even when the exponent fields are allowed to vary?}

The two-input case was recently settled in \cite{Previous2D}, where the rigidity result and the corresponding classification were established for variable-exponent Cobb--Douglas production functions. The extension of the rigidity phenomenon to more than two inputs was left open there. The purpose of the present paper is to resolve this genuinely higher-dimensional problem.

Let $U\subset\mathbb{R}_+^n$ be a connected open set. We consider
production functions
\[
Y(x)
=
A\prod_{\alpha=1}^{n}x_{\alpha}^{f_{\alpha}(x)},
\qquad x\in U,
\]
where $A>0$ and each exponent field
$f_{\alpha}\in C^{2}(U)$ is homogeneous of degree zero. This assumption means that the exponent fields depend only on the relative proportions of the production
factors and remain unchanged under a common rescaling of all inputs. In general, when the exponents vary with the inputs, they do not coincide with the differential output elasticities. The sum
\[
r(x)
=
\sum_{\alpha=1}^{n}f_{\alpha}(x)
\]
governs the exact response to a common rescaling of all inputs, since $Y(tx)=t^{r(x)}Y(x)$ whenever $x,tx\in U$. We therefore call $r=r(x)$ the local return-to-scale index, which in general need not be constant.

Our main theorem is stated for $n>2$, which is precisely the new case
addressed here. Intrinsic flatness of the graph hypersurface forces $r$ to be constant. Consequently, every flat
variable-returns-to-scale Cobb--Douglas production function in the
considered class is homogeneous. Thus the geometric rigidity problem is reduced to the homogeneous case. We subsequently classify the resulting homogeneous flat models. For degrees different from one, flatness yields the rigid power-of-a-linear-form family. The degree-one case is genuinely dimension-sensitive: for $n>2$ linear homogeneity alone is not sufficient for flatness, and the flat rank-one branch is described locally as an envelope of a one-parameter family of linear forms. This corrects the higher-dimensional degree-one converse in Theorem~1.1 of \cite{CV2013} and the corresponding branch of Theorem~4.2 therein, while leaving the non-degree-one and two-input results unaffected. Combining the rigidity and homogeneous results gives a complete classification of the flat variable-exponent production functions studied here when the resulting homogeneity degree is different from one, together with a local structural description of the degree-one branch.

The proof of the rigidity theorem is based on an affine logarithmic reduction principle. After
introducing logarithmic ratio coordinates \(u\) and a logarithmic scale
variable \(v\), the production function admits the structural
representation
\[
\ln Y=H(u)+R(u)v,
\]
where \(R=R(u)\) is the logarithmic form of the local return-to-scale
index. The two-dimensional logarithmic rigidity mechanism underlying this
representation is already present in \cite{Previous2D}; we isolate it below
in an autonomous lemma for use as an input. The new step is to construct,
for every logarithmic ratio direction, a suitable affine two-dimensional
restriction of the original \(n\)-variable function. The rank-one Hessian
condition is inherited by these affine restrictions, and their logarithmic
ratio coordinates freeze all but one component of \(u\). The two-dimensional
rigidity result therefore implies \(\partial_iR=0\) in each ratio direction,
and hence \(\nabla R=0\).

The paper is organized as follows. Section~2 introduces
variable-returns-to-scale Cobb--Douglas production functions and
establishes their scaling properties. Section~3 develops the logarithmic
coordinate representation. Section~4 gives the geometric
characterization of flat graph hypersurfaces and proves that flatness is stable under affine restrictions. Section~5 establishes the multidimensional
rigidity theorem, while Section~6 discusses the homogeneous case.

\section{Cobb--Douglas Production Functions with Variable Returns to Scale}\label{S2}

Let
\(
U\subset\mathbb{R}_{+}^{n}
\)
be a connected open set. Throughout the paper, we consider production functions
\(
Y:U\longrightarrow\mathbb{R}_{+}
\)
of the form
\begin{equation}\label{generalCD}
Y(x)
=
A\prod_{\alpha=1}^{n}
x_{\alpha}^{f_{\alpha}(x)},
\qquad
A>0,
\end{equation}
where
\(
f_{\alpha}\in C^{2}(U),
\)
\(
\alpha=1,\ldots,n.
\)

We assume that every exponent field is homogeneous of degree zero,
namely,
\begin{equation}\label{homdegreezero}
f_{\alpha}(tx)=f_{\alpha}(x),
\end{equation}
whenever
\(
x,tx\in U,
\)
\(
t>0.
\)

Throughout this paper, homogeneity on a general open set is understood in the same relative sense: a function $F:U\to\mathbb R$ is homogeneous of degree $\beta$ if
\[
F(tx)=t^\beta F(x)
\]
whenever $x,tx\in U$ and $t>0$, without assuming that \(U\) is dilation invariant.

The assumption \eqref{homdegreezero} means that the exponent fields depend only on the
relative proportions of the production factors and remain unchanged
under a common rescaling of all production inputs. We emphasize that, when the exponent fields depend on $x$, they are not in general equal to the differential output elasticities. Indeed,
\[
\frac{x_i}{Y}\frac{\partial Y}{\partial x_i}
=f_i(x)+x_i\sum_{\alpha=1}^n\frac{\partial f_\alpha}{\partial x_i}(x)\ln x_\alpha.
\]
Accordingly, we use the terms \emph{exponent fields} or \emph{variable exponents} for the functions $f_\alpha$.

\begin{definition}\cite{Janvry}
A production function of the form
\eqref{generalCD},
whose exponent fields satisfy
\eqref{homdegreezero},
is called a
\emph{variable-returns-to-scale Cobb--Douglas production function}.

\end{definition}

\begin{remark}

If the exponent fields are constant,
the model reduces to the classical homogeneous Cobb--Douglas production
function \cite{CD}.

\end{remark}

For a variable-returns-to-scale Cobb--Douglas production function \(Y\) given by \eqref{generalCD}, we set
\begin{equation}\label{r-t-s-index}
r(x)
=
\sum_{\alpha=1}^{n}
f_{\alpha}(x),
\end{equation}
the local return-to-scale index.

The next proposition characterizes the homogeneous members of the
considered class of production functions.

\begin{proposition}\label{prop:homogeneous-characterization}

A variable-returns-to-scale Cobb--Douglas production function is
homogeneous if and only if its local return-to-scale index is constant.

\end{proposition}

\begin{proof}
If $r(x)\equiv\beta$, then the degree-zero homogeneity of the exponent
fields gives
\[
Y(tx)=t^{r(x)}Y(x)=t^\beta Y(x)
\]
whenever $x,tx\in U$ and $t>0$. Hence $Y$ is homogeneous of degree
$\beta$.

Conversely, suppose that $Y$ is homogeneous of degree $\beta$. Fix $x\in U$.
Since $U$ is open, $tx\in U$ for all $t$ in some open interval containing $1$.
The degree-zero homogeneity of the exponent fields and the homogeneity of $Y$
give
\[
t^{r(x)}Y(x)=Y(tx)=t^\beta Y(x).
\]
Because $Y(x)>0$, this implies $t^{r(x)}=t^\beta$ for all such $t$, hence
$r(x)=\beta$. Since $x$ was arbitrary, $r$ is constant.
\end{proof}

The above results show that the entire deviation from homogeneity is
encoded in the local return-to-scale index.
Consequently, the rigidity theorem proved later may be interpreted as
showing that the flatness of the graph hypersurface eliminates every
possible deviation from homogeneity.

\section{Logarithmic Coordinates and Structural Representation}\label{S3}

The degree-zero homogeneity of the exponent fields introduced in the
previous section naturally leads to logarithmic coordinates that separate
the dependence on the proportions of the production factors from the
dependence on their common scale.

Let
\(
U\subset\mathbb R_+^n
\)
be the domain of the production function.
For
\(
x=(x_1,\ldots,x_n)\in U,
\)
introduce the logarithmic ratio coordinates
\[
u_i
=
\ln\frac{x_i}{x_n},
\qquad
i=1,\ldots,n-1,
\]
together with the logarithmic scale variable
\[
v=\ln x_n.
\]

The inverse transformation is given by
\[
x_i=e^{u_i+v},
\qquad
i=1,\ldots,n-1,
\]
and
\[
x_n=e^v.
\]

Thus,
\[
(x_1,\ldots,x_n)
\longmapsto
(u_1,\ldots,u_{n-1},v)
\]
defines a smooth diffeomorphism between \(U\) and an open connected set
\(
\Omega\subset\mathbb R^n.
\)

Let
\[
\pi:\Omega\longrightarrow\mathbb R^{n-1},
\qquad
\pi(u,v)=u,
\]
denote the canonical projection, and define the
\emph{logarithmic ratio domain} by
\(
D=\pi(\Omega).
\)

Since \(\Omega\) is connected and \(\pi\) is continuous,
\(D\) is connected. Moreover, since \(\pi\) is an open map,
\(D\) is open. Therefore,
\(D\) is an open connected subset of
\(\mathbb R^{n-1}\).

The following proposition shows that each exponent field depends only on
the logarithmic ratio coordinates.

\begin{proposition}\label{prop:F}
For every
\(
\alpha=1,\ldots,n,
\)
there exists a unique function
\(
F_\alpha\in C^2(D)
\)
such that
\(
f_\alpha(x)
=
F_\alpha(u),
\)
where
\[
u
=
\left(
\ln\frac{x_1}{x_n},
\ldots,
\ln\frac{x_{n-1}}{x_n}
\right).
\]
\end{proposition}

\begin{proof}

For
\(
(u,v)\in\Omega,
\)
write
\[
x(u,v)
=
\left(
e^{u_1+v},
\ldots,
e^{u_{n-1}+v},
e^v
\right).
\]

Fix
\(
u\in D.
\)
If
\(
(u,v_1),(u,v_2)\in\Omega,
\)
then
\[
x(u,v_2)
=
e^{\,v_2-v_1}x(u,v_1).
\]

Since each exponent field is homogeneous of degree zero,
\[
f_\alpha(x(u,v_2))
=
f_\alpha(x(u,v_1)).
\]
By the relative notion of homogeneity used here, it is enough that both
$x(u,v_1)$ and $x(u,v_2)$ belong to $U$.
Hence the function
\[
F_\alpha(u):=f_\alpha(x(u,v))
\]
is independent of the choice of $v$, and therefore is well defined.

It remains to verify the regularity of $F_\alpha$. Fix $u^0\in D$ and
choose $v^0$ such that $(u^0,v^0)\in\Omega$. Since $\Omega$ is open,
there exists an open neighborhood $B$ of $u^0$ such that
\(
B\times\{v^0\}\subset\Omega.
\)
On $B$ we may use the fixed representative $v^0$, so that
\[
F_\alpha(u)=f_\alpha(x(u,v^0)),\qquad u\in B.
\]
The map $u\mapsto x(u,v^0)$ is smooth and $f_\alpha\in C^2(U)$;
hence $F_\alpha|_B\in C^2(B)$. Since $u^0$ was arbitrary,
it follows that $F_\alpha\in C^2(D)$.

\end{proof}

Define
\(
R:D\longrightarrow\mathbb R
\)
by
\begin{equation}\label{functionR}
R(u)
=
\sum_{\alpha=1}^{n}
F_\alpha(u).
\end{equation}
Then,
by Proposition~\ref{prop:F}, we derive
\begin{equation}\label{return-index-log}
R
\left(
\ln\frac{x_1}{x_n},
\ldots,
\ln\frac{x_{n-1}}{x_n}
\right)
=
r(x),
\qquad
x\in U,
\end{equation}
where \(r\) denotes the local return-to-scale index defined by
\eqref{r-t-s-index}.

The following theorem is the structural result on which the remainder
of the paper is based.

\begin{theorem}
\label{thm:logrepresentation}

There exists a unique function
\(
H\in C^2(D)
\)
such that
\begin{equation}\label{rigidity-log-form}
\ln Y
=
H(u)+R(u)v.
\end{equation}

\end{theorem}

\begin{proof}

Starting from
\[
Y
=
A
\prod_{\alpha=1}^{n}
x_\alpha^{f_\alpha(x)},
\]
we obtain
\begin{equation}\label{forma-log}
\ln Y
=
\ln A
+
\sum_{\alpha=1}^{n}
F_\alpha(u)\ln x_\alpha.
\end{equation}

Since
\[
\ln x_i=u_i+v,
\qquad
i=1,\ldots,n-1,
\]
and
\[
\ln x_n=v,
\]
we have from \eqref{forma-log}
\[
\begin{aligned}
\ln Y
&=
\ln A
+
\sum_{i=1}^{n-1}
F_i(u)(u_i+v)
+
F_n(u)v
\end{aligned}
\]
and this implies
\begin{equation}\label{reppart}
\begin{aligned}
\ln Y
&=
\ln A
+
\sum_{i=1}^{n-1}
u_iF_i(u)
+
\left(
\sum_{\alpha=1}^{n}
F_\alpha(u)
\right)v.
\end{aligned}
\end{equation}

Define now
\begin{equation}\label{functionH}
H(u)
=
\ln A
+
\sum_{i=1}^{n-1}
u_iF_i(u).
\end{equation}
Then the required representation \eqref{rigidity-log-form} follows immediately from \eqref{reppart}, taking account of \eqref{functionR} and \eqref{functionH}. The uniqueness of \(H\) is evident.

\end{proof}

\begin{remark}

The logarithmic representation obtained in Theorem \ref{thm:logrepresentation} completely separates the variables
describing the proportions of the production factors from the variable
describing their common scale. The dependence on the logarithmic scale variable is affine, and its
coefficient is precisely the local return-to-scale index.
Consequently, every deviation from homogeneity is entirely encoded in
the function \(R\).

\end{remark}

\begin{remark}

The logarithmic representation is independent of the dimension.
The dimension of the problem is reflected only in the number of
logarithmic ratio coordinates.

\end{remark}

\section{Geometric Characterization of Flat Graph Hypersurfaces}\label{S4}

The logarithmic representation established in the previous section
provides the analytical structure of the production function.
The purpose of the present section is to develop the geometric
ingredient needed for the proof of the rigidity theorem.

We begin by recalling a standard characterization of flat graph hypersurfaces in terms of the rank of the Hessian matrix. We then prove that flatness is preserved under affine restrictions. This reduction provides the geometric mechanism that allows the multidimensional rigidity problem to be decomposed into a family of two-dimensional rigidity problems.

Let
\(
Z\in C^{2}(U),
\)
\(U\subset\mathbb R^{n},
\)
and consider its graph hypersurface
\[
M
=
\{
(x,Z(x))
:
x\in U
\}
\subset\mathbb R^{n+1}.
\]

The following result is standard.

\begin{proposition}
\label{prop:flatness}

The graph hypersurface of \(Z\) is flat if and only if the Hessian matrix $D^{2}Z$ of \(Z\) satisfies
\[
\operatorname{rank}D^{2}Z\le1.
\]

\end{proposition}

\begin{proof}

Writing
\[
Z_i=\frac{\partial Z}{\partial x_i},
\qquad
Z_{ij}=\frac{\partial^2Z}{\partial x_i\partial x_j},
\]
the induced metric of the graph hypersurface is
\[
g_{ij}
=
\delta_{ij}
+
Z_iZ_j,
\]
while the coefficients of the second fundamental form are
\[
b_{ij}
=
\frac{Z_{ij}}
{\sqrt{1+|\nabla Z|^2}}.
\]

Since the ambient space
\(
\mathbb R^{n+1}
\)
is Euclidean,  the Gauss equation gives the components
\(R_{ijkl}\)
of the Riemann curvature tensor of the induced metric:
\[
R_{ijkl}
=
b_{ik}b_{jl}
-
b_{il}b_{jk}.
\]

Therefore,
\begin{equation}\label{Rijkl}
R_{ijkl}
=
\frac{
Z_{ik}Z_{jl}
-
Z_{il}Z_{jk}
}
{
1+|\nabla Z|^2
}.
\end{equation}

By definition, \(M\) is flat if and only if
\(
R_{ijkl}=0
\)
for all indices \(i,j,k,l\). In view of \eqref{Rijkl},
this is equivalent to
\[
Z_{ik}Z_{jl}
=
Z_{il}Z_{jk}.
\]

The above identities are equivalent to the vanishing of all
\(2\times2\) minors of the Hessian matrix, and therefore
\[
\operatorname{rank}D^2Z\le1.
\]
This completes the proof.

\end{proof}

The next result shows that the rank-one Hessian condition is stable under affine restrictions.

\begin{lemma}\label{lem:affine}

Let
\(
U\subset\mathbb R^{n}
\)
be open,
let
\(
Z\in C^{2}(U),
\)
and let
\(
V\subset\mathbb R^{m}
\)
be open.
Suppose that
\(
L:V\longrightarrow U
\)
is the restriction of an affine map.
If
\[
\operatorname{rank}D^{2}Z(x)\le1,
\qquad
x\in U,
\]
then
\[
\operatorname{rank}
D^{2}(Z\circ L)(y)
\le1,
\qquad
y\in V.
\]

\end{lemma}

\begin{proof}

Write
\[
L(y)=Ay+b,
\]
where
\(
A\in\mathbb R^{n\times m}
\)
is constant and
\(
b\in\mathbb R^{n}.
\)

Since \(L\) is affine, its second differential vanishes.
Hence, by the chain rule for second derivatives,
\[
D^{2}(Z\circ L)(y)
=
A^{T}
D^{2}Z(L(y))
A.
\]

Since
\[
\operatorname{rank}(A^{T}BA)
\le
\operatorname{rank}(B)
\]
for every matrix \(B\), it follows that
\[
\operatorname{rank}
D^{2}(Z\circ L)(y)
\le
\operatorname{rank}
D^{2}Z(L(y))
\le1.
\]

This proves the lemma.

\end{proof}

As an immediate consequence, for every affine map $L:V\subset\mathbb R^2\to U$, if the graph of $Z$ is flat, then the graph surface of the restricted function $Z\circ L$ is flat.
Together with Proposition~\ref{prop:flatness} and
Theorem~\ref{thm:logrepresentation}, this affine restriction principle
provides the basic reduction used in the next section.

\section{Higher-Dimensional Rigidity}\label{S5}
\label{sec:rigidity}

Throughout the section, suppose
\(
n>2
\)
and let
\[
Y(x)
=
A\prod_{\alpha=1}^{n}
x_\alpha^{f_\alpha(x)}
\]
be a variable-returns-to-scale Cobb--Douglas production function defined
on a connected open set
\(
U\subset\mathbb R_+^n.
\)

In the logarithmic coordinates introduced in Section~3, we have
\[
u_i=\ln\frac{x_i}{x_n},
\qquad
i=1,\ldots,n-1,
\qquad
v=\ln x_n,
\]
Theorem~\ref{thm:logrepresentation} gives the representation
\eqref{rigidity-log-form},
where
\(
u=(u_1,\ldots,u_{n-1})\in D
\)
and
\(
R(u)
\)
is the logarithmic representation \eqref{functionR} of the local return-to-scale index.

Our objective is to prove that
\(
\nabla R=0
\)
on the logarithmic ratio domain \(D\) defined in Section 3, provided that the hypersurface associated to the production function $Y$ is flat.

\subsection{Two-dimensional logarithmic rigidity}\label{S5.1}

We recall the two-dimensional analytic rigidity mechanism established
in \cite{Previous2D}. We state it here in a form adapted to the affine
restrictions used in the higher-dimensional argument.

\begin{lemma}
\label{lem:2d-log-rigidity}
Let
\(
V\subset\mathbb R_+^2
\)
be a connected open set, and let
\(
G\in C^2(V)
\)
be positive. Define
\[
I
=
\left\{
\ln\frac{x}{y}:(x,y)\in V
\right\}.
\]
Suppose that there exist functions
\(
\mathcal H,\rho\in C^2(I)
\)
such that
\[
\ln G(x,y)
=
\mathcal H(s)+\rho(s)w,
\qquad
s=\ln\frac{x}{y},
\quad
w=\ln y.
\]
If
\(
\det D^2G=0
\)
on $V$, then $\rho$ is constant on $I$.
\end{lemma}

\begin{proof}
Set $g=\ln G$. Since
\[
D^2_{x,y}G
=
e^g\left(
D^2_{x,y}g+\nabla_{x,y}g\,\nabla_{x,y}g^T
\right),
\]
the condition $\det D^2G=0$ is equivalent to
\[
\det\left(
D^2_{x,y}g+\nabla_{x,y}g\,\nabla_{x,y}g^T
\right)=0.
\]
For the representation
\[
g=\mathcal H(s)+\rho(s)w,
\qquad
s=\ln\frac{x}{y},
\quad
w=\ln y,
\]
the computation carried out in \cite{Previous2D} gives
\[
\det\left(
D^2_{x,y}g+\nabla_{x,y}g\,\nabla_{x,y}g^T
\right)
=
e^{-2s-4w}
\left(
P_2(s)w^2+P_1(s)w+P_0(s)
\right),
\]
where
\[
P_2=(\rho-1)(\rho')^2,
\]
\[
P_1
=
(\rho-1)
\left[
\rho\rho''
+(2\mathcal H'-\rho)\rho'
-2(\rho')^2
\right],
\]
and
\[
P_0
=
(\rho-1)
\left[
(\mathcal H')^2
+\rho(\mathcal H''-\mathcal H')
-2\mathcal H'\rho'
\right]
-(\rho')^2.
\]

Let
\[
\Phi(x,y)
=
\left(\ln\frac{x}{y},\ln y\right).
\]
Since $\Phi$ is a diffeomorphism of $\mathbb R_+^2$ onto
$\mathbb R^2$, the set $\Phi(V)$ is open. Thus, for each $s\in I$,
the set
\[
W_s
=
\{w\in\mathbb R:(s,w)\in\Phi(V)\}
\]
is nonempty and open, and hence contains a nonempty open interval.
The polynomial
\[
P_2(s)w^2+P_1(s)w+P_0(s)
\]
therefore vanishes on an open interval and must be identically zero.
In particular,
\[
(\rho(s)-1)(\rho'(s))^2=0
\]
for every $s\in I$.

Since $V$ is connected, $I$ is an interval. Set
\[
E=\{s\in I:\rho(s)\neq1\}.
\]
Then $\rho'=0$ on $E$, so $\rho$ is constant on each connected
component of $E$. If a nonempty component were a proper subset of
$I$, continuity at a boundary point in $I$ would force its constant
value to be $1$, contradicting the definition of $E$. Hence either
$E=\varnothing$ or $E=I$. In the first case $\rho\equiv1$, while in
the second $\rho'\equiv0$ on $I$. Thus $\rho$ is constant on $I$.
\end{proof}

\subsection{Restrictions to coordinate lines}\label{S5.2}

We record an elementary observation that will be used to identify the
derivatives obtained from the two-dimensional restrictions with the
partial derivatives of \(R\). Here and below, "$\cdot$" denotes the standard Euclidean inner product.

\begin{lemma}
\label{lem:coordinate-lines}

Let
\(
D\subset\mathbb R^{n-1}
\)
be open, let
\(
R\in C^2(D),
\)
and fix
\(
u^0=(u_1^0,\ldots,u_{n-1}^0)\in D
\)
and
\(
i\in\{1,\ldots,n-1\}.
\)
Define
\[
\gamma_{i,u^0}(s)
=
(u_1^0,\ldots,u_{i-1}^0,s,
u_{i+1}^0,\ldots,u_{n-1}^0),
\qquad s\in\mathbb R.
\]

Let \(J_{i,u^0}\) be the connected component containing \(u_i^0\) of
the open set
\[
\gamma_{i,u^0}^{-1}(D)
=
\{\,s\in\mathbb R:\gamma_{i,u^0}(s)\in D\,\}.
\]

Then \(J_{i,u^0}\) is an open interval, and the function
\[
R_{i,u^0}
:=
R\circ\gamma_{i,u^0}
\]
belongs to \(C^2(J_{i,u^0})\). Moreover,
\[
R_{i,u^0}'(s)
=
\frac{\partial R}{\partial u_i}
\left(
\gamma_{i,u^0}(s)
\right)
\]
for every \(s\in J_{i,u^0}\). In particular,
\[
R_{i,u^0}'(u_i^0)
=
\frac{\partial R}{\partial u_i}(u^0).
\]

\end{lemma}

\begin{proof}

Since
\[
\gamma_{i,u^0}(u_i^0)=u^0\in D,
\]
the set
\(
\gamma_{i,u^0}^{-1}(D)
\)
contains \(u_i^0\). Moreover, since \(D\) is open and
\(\gamma_{i,u^0}\) is continuous,
\(
\gamma_{i,u^0}^{-1}(D)
\)
is open in \(\mathbb R\). Therefore, its connected component
\(J_{i,u^0}\) containing \(u_i^0\) is an open interval.

Since
\(
R\in C^2(D)
\)
and
\(
\gamma_{i,u^0}
\)
is smooth,
\[
R_{i,u^0}
=
R\circ\gamma_{i,u^0}
\]
belongs to
\(
C^2(J_{i,u^0}).
\)

By the chain rule,
\[
R_{i,u^0}'(s)
=
\nabla R\!\left(\gamma_{i,u^0}(s)\right)\cdot
\gamma_{i,u^0}'(s).
\]

Since
\[
\gamma_{i,u^0}'(s)=e_i,
\]
where
\(
e_i
\)
denotes the \(i\)-th standard basis vector of
\(
\mathbb R^{n-1},
\)
we obtain
\[
R_{i,u^0}'(s)
=
\nabla R\!\left(\gamma_{i,u^0}(s)\right)\cdot e_i
=
\frac{\partial R}{\partial u_i}
\left(
\gamma_{i,u^0}(s)
\right).
\]

Evaluating at
\(
s=u_i^0
\)
and using
\[
\gamma_{i,u^0}(u_i^0)=u^0,
\]
we obtain
\[
R_{i,u^0}'(u_i^0)
=
\frac{\partial R}{\partial u_i}(u^0).
\]

\end{proof}

\subsection{Affine logarithmic sections}

Fix
\(
u^0=(u_1^0,\ldots,u_{n-1}^0)\in D
\)
and
\(
i\in\{1,\ldots,n-1\}.
\)
For every
\(
j\in\{1,\ldots,n-1\}\setminus\{i\},
\)
set
\(
c_j=e^{u_j^0}.
\)
Define now the affine map
\(
L_{i,u^0}:\mathbb R^2\longrightarrow\mathbb R^n
\)
by
\[
L_{i,u^0}(x,y)
=
(z_1,\ldots,z_n),
\]
where
\(
z_i=x,
\)
\(
z_n=y,
\)
and
\(
z_j=c_jy
\)
$(j\neq i,n)$.

Let
\(
V_{i,u^0}
=
L_{i,u^0}^{-1}(U)\cap\mathbb R_+^2
\)
and define
\[
Y_{i,u^0}
=
Y\circ L_{i,u^0}.
\]

Since \(U\) is open and \(L_{i,u^0}\) is continuous,
\(V_{i,u^0}\) is open.

\begin{proposition}\label{prop:affine-log-sections}
Suppose that the graph hypersurface of $Y$ is flat. Fix $u^0\in D$, $i\in\{1,\ldots,n-1\}$, and choose $v^0$ such that $(u^0,v^0)\in\Omega$. Set
\(
p^0=\left(e^{u_i^0+v^0},e^{v^0}\right)\in V_{i,u^0}.
\)
Then there exists a connected open neighborhood $W_0$ of $p^0$, with $W_0\subset V_{i,u^0}$, such that
\[
I_{W_0}:=\left\{\ln\frac{x}{y}:(x,y)\in W_0\right\}\subset J_{i,u^0}.
\]
On $W_0$ one has
\[
\det D^2Y_{i,u^0}=0
\]
and
\begin{equation}\label{section-log-representation}
\ln Y_{i,u^0}(x,y)=H_{i,u^0}(s)+R_{i,u^0}(s)w,
\end{equation}
where $s=\ln(x/y)$, $w=\ln y$, $H_{i,u^0}=H\circ\gamma_{i,u^0}$ and $R_{i,u^0}=R\circ\gamma_{i,u^0}$ on $I_{W_0}$.
\end{proposition}

\begin{proof}
By Proposition~\ref{prop:flatness}, $\operatorname{rank}D^2Y\le1$ on $U$. Since $L_{i,u^0}$ is affine, Lemma~\ref{lem:affine} gives
\[
\operatorname{rank}D^2Y_{i,u^0}\le1
\]
on $V_{i,u^0}$, and hence $\det D^2Y_{i,u^0}=0$ there.

Because $J_{i,u^0}$ is an open interval containing $u_i^0$, choose $\varepsilon>0$ such that
\(
(u_i^0-\varepsilon,u_i^0+\varepsilon)\subset J_{i,u^0}.
\)
The map $(x,y)\mapsto\ln(x/y)$ is continuous and takes the value $u_i^0$ at $p^0$. Since $V_{i,u^0}$ is open, there exists a connected open neighborhood $W_0$ of $p^0$, contained in $V_{i,u^0}$, for which
\[
I_{W_0}\subset(u_i^0-\varepsilon,u_i^0+\varepsilon)\subset J_{i,u^0}.
\]

For $(x,y)\in W_0$, the point $L_{i,u^0}(x,y)$ has logarithmic ratio coordinates $\gamma_{i,u^0}(s)$ and logarithmic scale coordinate $w$. Therefore Theorem~\ref{thm:logrepresentation} gives
\[
\ln Y_{i,u^0}(x,y)
=H(\gamma_{i,u^0}(s))+R(\gamma_{i,u^0}(s))w,
\]
which is exactly \eqref{section-log-representation}. The inclusion $I_{W_0}\subset J_{i,u^0}$ guarantees that both compositions are defined on the whole ratio set used here.
\end{proof}

\subsection{The rigidity theorem}\label{S5.4}

Now, we are able to establish the main result. The proof proceeds in
three steps. First, we restrict the production function to suitable
affine two-dimensional subspaces of its domain. Second, Proposition
\ref{prop:affine-log-sections} shows that each restriction admits the
same logarithmic representation as in the two-dimensional setting.
Finally, the two-dimensional rigidity lemma implies that every
directional derivative of the return-to-scale function \(R\) along the
coordinate lines vanishes, forcing \(R\) to be constant on the entire
logarithmic ratio domain.

\begin{theorem}
\label{thm:main-rigidity}

Let
\[
Y(x)
=
A\prod_{\alpha=1}^{n}
x_\alpha^{f_\alpha(x)},
\qquad
n>2,
\]
be a variable-returns-to-scale Cobb--Douglas production function defined
on a connected open set
\(
U\subset\mathbb R_+^n.
\)
If the graph hypersurface of \(Y\) is flat, then there exists a constant
\(
\beta\in\mathbb R
\)
such that
\[
\sum_{\alpha=1}^{n}f_\alpha(x)
=
\beta
\]
for every \(x\in U\).

\end{theorem}

\begin{proof}
It is enough to prove that $\partial R/\partial u_i=0$ throughout $D$ for every $i=1,\ldots,n-1$. Fix $u^0\in D$ and $i\in\{1,\ldots,n-1\}$. Choose $v^0$ with $(u^0,v^0)\in\Omega$. By Proposition~\ref{prop:affine-log-sections}, there is a connected open neighborhood $W_0\subset V_{i,u^0}$ of
\(
p^0=\left(e^{u_i^0+v^0},e^{v^0}\right)
\)
such that $I_{W_0}\subset J_{i,u^0}$, $\det D^2Y_{i,u^0}=0$ on $W_0$, and
\[
\ln Y_{i,u^0}(x,y)=H_{i,u^0}(s)+R_{i,u^0}(s)w.
\]
All hypotheses of Lemma~\ref{lem:2d-log-rigidity} are therefore satisfied on $W_0$. Hence $R_{i,u^0}$ is constant on $I_{W_0}$. Since
\[
u_i^0=\ln\frac{e^{u_i^0+v^0}}{e^{v^0}}\in I_{W_0},
\]
we obtain $R_{i,u^0}'(u_i^0)=0$. Lemma~\ref{lem:coordinate-lines} now yields
\[
\frac{\partial R}{\partial u_i}(u^0)=0.
\]
Because $u^0$ and $i$ were arbitrary, $\nabla R=0$ on $D$. The domain $D$ is connected, so $R\equiv\beta$ for some constant $\beta$. Finally, by \eqref{return-index-log},
\[
\sum_{\alpha=1}^n f_\alpha(x)=r(x)=\beta
\]
for every $x\in U$.
\end{proof}

\begin{corollary}
\label{cor:homogeneity}

Every flat variable-returns-to-scale Cobb--Douglas production function
with $n>2$ inputs is homogeneous.

\end{corollary}

\begin{proof}

By Theorem~\ref{thm:main-rigidity}, the local return-to-scale index is
constant. The conclusion follows from
Proposition~\ref{prop:homogeneous-characterization}.

\end{proof}

The Rigidity Theorem shows that flatness eliminates the only
possible obstruction to homogeneity, namely the variability of the local
return-to-scale index.

\section{Higher-Dimensional Homogeneous Classification}
\label{sec:classification}

Theorem~\ref{thm:main-rigidity} reduces the problem considered in this
paper to the homogeneous case: every flat production function
in the present variable-exponent class has a constant local return-to-scale
index and is therefore homogeneous in the relative-domain sense specified
in Section~2. The classification obtained in \cite{CV2013} is the
natural point of departure, but its stated ``if and only if'' degree-one
branch requires a dimensional correction.
For \(n=2\), degree-one
homogeneity indeed implies flatness. For \(n>2\), it does not:
Euler's identity gives only \(D^2F(x)x=0\), which need not force
\(\operatorname{rank}D^2F\le1\). We therefore separate the valid
non-degree-one classification from the exceptional degree-one branch and
give a corrected higher-dimensional formulation. We first record a
sufficient condition under which the two-dimensional converse remains
valid in arbitrary ambient dimension.

\begin{proposition}\label{prop:effective-two}
Let $V\subset\mathbb R^n$ be open and suppose that
\(
 F(x)=G(Ax),
\)
$x\in\mathbb R^n$,
where $A:\mathbb R^n\to\mathbb R^2$ is linear with $\operatorname{rank}A\le2$, $G\in C^2(W)$ on an open set $W\subset\mathbb R^2$, $A(V)\subset W$, and $G$ is homogeneous of degree one in the relative-domain sense. Then the graph of $F$ is flat on $V$.
\end{proposition}

\begin{proof}
By Euler's identity for $G$, we have
\[
D^2G(z)\,z=0.
\]
If $\operatorname{rank}A\le1$, then
\[
 \operatorname{rank}D^2F(x)
 =\operatorname{rank}\bigl(A^TD^2G(Ax)A\bigr)
 \le \operatorname{rank}A\le1.
\]
Assume now that $\operatorname{rank}A=2$. If $z=Ax\ne0$, the $2\times2$ Hessian $D^2G(z)$ has the nonzero vector $z$ in its kernel, and hence
\[
 \operatorname{rank}D^2G(z)\le1.
\]
If $z=0$, then $0\in W$. Since $W$ is open, a ball $B$ centered at the origin is contained in $W$. Relative degree-one homogeneity gives $G(0)=0$ and, for $y\in B$ and sufficiently small positive $t$,
\[
 G(ty)=tG(y).
\]
Letting $t\to0^+$ and using differentiability at the origin yields $G(y)=DG(0)\cdot y$ on $B$. Thus $D^2G(0)=0$. Consequently,
\[
 \operatorname{rank}D^2G(Ax)\le1
\]
for every $x\in V$. The chain rule gives
\[
 D^2F(x)=A^TD^2G(Ax)A,
\]
so
\[
 \operatorname{rank}D^2F(x)\le \operatorname{rank}D^2G(Ax)\le1.
\]
Proposition~\ref{prop:flatness} therefore implies that the graph of $F$ is flat.
\end{proof}

\begin{remark}
Proposition~\ref{prop:effective-two} contains the genuine two-input case as a special case, but it also applies in arbitrary ambient dimension whenever the production technology depends effectively on at most two independent linear combinations of the inputs. Thus the linearly homogeneous converse remains valid under this additional effective-dimensional hypothesis. For genuinely higher-dimensional models, Euler's identity supplies only the radial null direction of the Hessian and no longer forces the rank-one condition.
\end{remark}

For a $C^2$ homogeneous function $F$ of degree $\beta$, Euler's identity and its differential give
\begin{equation}\label{euler-hessian}
 x\cdot\nabla F(x)=\beta F(x),\qquad
 D^2F(x)x=(\beta-1)\nabla F(x).
\end{equation}
Before stating the classification, observe that the degree-zero case cannot occur under the nonconstancy assumption if the graph is flat. Indeed, if $\beta=0$ and $\nabla F(p)\ne0$ at some point, then on a small connected neighborhood $V$ of $p$ equation~\eqref{euler-hessian} gives
\[
 D^2F(x)x=-\nabla F(x)\ne0.
\]
Flatness therefore forces $\operatorname{rank}D^2F=1$ on $V$ and
\[
 \operatorname{Im}D^2F(x)=\operatorname{span}\{\nabla F(x)\}.
\]
The normalized-gradient argument used below shows that the direction of $\nabla F$ is constant on $V$, so $\nabla F=q c$ with $c\ne0$ fixed and $q$ nowhere zero. Euler's identity then gives $c\cdot x=0$ on the open set $V$, a contradiction. Hence $\nabla F\equiv0$, and connectedness makes $F$ constant. Thus no degree-zero branch occurs for the nonconstant production functions considered here.

We can now state the homogeneous result in a form that gives a complete classification away from the exceptional degree-one case and a local structural description of that branch.

\begin{theorem}\label{thm:corrected-homogeneous-classification}
Let $U\subset\mathbb R_+^n$ be a connected open set and let $F\in C^2(U)$ be positive, nonconstant, and homogeneous of degree $\beta\ne0$.

\begin{enumerate}
\item If $\beta\ne1$, then the graph of $F$ is flat if and only if there exist a nonzero constant vector $c\in\mathbb R^n$ and a constant $C>0$, with $c\cdot x>0$ on $U$, such that
\begin{equation}\label{power-linear-class}
 F(x)=C(c\cdot x)^\beta.
\end{equation}

\item Let $\beta=1$ and suppose that the graph of $F$ is flat. On the open rank-one set
\[
 U_1=\{x\in U:\operatorname{rank}D^2F(x)=1\},
\]
every point has a neighborhood $V$ on which there exist an interval $I$, a $C^1$ regular curve $a:I\to\mathbb R^n$, and a $C^1$ function $\tau:V\to I$ such that
\begin{equation}\label{envelope-class}
 F(x)=a(\tau(x))\cdot x,
 \qquad
 a'(\tau(x))\cdot x=0.
\end{equation}
On every connected open subset on which $D^2F=0$, the function $F$ is linear.

Conversely, if a $C^2$ degree-one homogeneous function admits the representation \eqref{envelope-class} on an open set $V$, then its graph is flat on $V$. A linear homogeneous function also has a flat graph.
\end{enumerate}
\end{theorem}

\begin{proof}
Suppose first that $\beta\ne1$ and that the graph of $F$ is flat. Since $\beta\ne0$ and $F>0$, Euler's identity gives
\[
 x\cdot\nabla F(x)=\beta F(x)\ne0,
\]
so $\nabla F$ never vanishes. Equation~\eqref{euler-hessian} then shows that $D^2F$ never vanishes. Hence, we have
\(
 \operatorname{rank}D^2F=1
\)
throughout $U$.

Equation~\eqref{euler-hessian} also gives
\[
 \operatorname{Im}D^2F(x)=\operatorname{span}\{\nabla F(x)\}
 \qquad (x\in U).
\]
Define the normalized gradient field
\[
 N(x)=\frac{\nabla F(x)}{\|\nabla F(x)\|}.
\]
For every $v\in\mathbb R^n$, the vector
\[
 D_v(\nabla F)(x)=D^2F(x)v
\]
is parallel to $\nabla F(x)$. Differentiating the normalization therefore gives $D_vN(x)=0$ for every $v$, hence $DN=0$ on $U$. Since $U$ is connected, $N$ is constant. Consequently, there is a fixed nonzero vector $c$ and a nowhere-zero scalar function $q$ such that
\[
 \nabla F=q(x)c.
\]
Euler's identity yields
\[
 q(x)(c\cdot x)=\beta F(x).
\]
In particular, $c\cdot x$ never vanishes on $U$; after replacing $c$ by $-c$ if necessary, we may assume $c\cdot x>0$ throughout $U$. Therefore
\[
 \nabla\log F
 =\frac{\nabla F}{F}
 =\beta\frac{c}{c\cdot x}
 =\nabla\bigl(\beta\log(c\cdot x)\bigr).
\]
Connectedness of $U$ gives
\[
 \log F-\beta\log(c\cdot x)=\log C
\]
for a constant $C>0$, proving~\eqref{power-linear-class}. Conversely, a function of the form~\eqref{power-linear-class} satisfies
\[
 D^2F=C\beta(\beta-1)(c\cdot x)^{\beta-2}cc^T,
\]
and therefore has Hessian rank at most one. Proposition~\ref{prop:flatness} gives flatness.

Finally, let $\beta=1$ and assume flatness. On $U_1$, the
gradient map \(\nabla F\) has constant rank one. After shrinking to a
neighborhood \(V\) of any \(p\in U_1\), the constant-rank theorem implies
that \(\nabla F(V)\) is a one-dimensional embedded \(C^1\) submanifold and
that the gradient map factors through a \(C^1\) submersion
\(\tau:V\to I\). Choosing a regular \(C^1\) parametrization
\(a:I\to\mathbb R^n\) of this image, we obtain
\[
 \nabla F(x)=a(\tau(x)).
\] Euler's identity gives
\[
 F(x)=x\cdot\nabla F(x)=a(\tau(x))\cdot x.
\]
Differentiating this expression and comparing it with $dF=a(\tau)\cdot dx$ gives
\[
 \bigl(a'(\tau(x))\cdot x\bigr)d\tau=0.
\]
Since $d\tau\ne0$ in the rank-one parametrization, we obtain
\[
 a'(\tau(x))\cdot x=0,
\]
which proves~\eqref{envelope-class}. On a connected open set on which $D^2F=0$, the gradient is constant, so $F=c\cdot x+d$. Degree-one homogeneity forces $d=0$, hence $F$ is linear.

For the local converse, suppose that a $C^2$ degree-one homogeneous function admits~\eqref{envelope-class} on $V$. Differentiation gives
\[
 dF=a(\tau)\cdot dx+\bigl(a'(\tau)\cdot x\bigr)d\tau=a(\tau)\cdot dx,
\]
so $\nabla F=a(\tau)$. Hence
\[
 D^2F=a'(\tau)\otimes\nabla\tau,
\]
which has rank at most one. Proposition~\ref{prop:flatness} shows that the graph is flat on $V$. The linear case is immediate.
\end{proof}

\begin{remark}\label{rem:two-input-classification}
When $n=2$ and $\beta=1$, equation~\eqref{euler-hessian} gives $D^2F(x)x=0$. Since $x\ne0$ and $D^2F$ is a $2\times2$ matrix, this automatically implies $\operatorname{rank}D^2F\le1$. Thus the degree-one condition itself is sufficient for flatness in two variables. Consequently, Theorem~\ref{thm:corrected-homogeneous-classification} reduces in the two-input case to the familiar alternatives
\[
 \beta=1
 \qquad\text{or}\qquad
 \beta\ne0,1,\quad F(x)=C(c_1x_1+c_2x_2)^\beta,
\]
in agreement with the recent two-input classification in \cite{Previous2D} and with the two-dimensional converse established in \cite{CV2013}.
\end{remark}

\begin{remark}
The distinction in Proposition~\ref{prop:effective-two} is sharp at the level of the rank argument. For instance, on $\mathbb R_+^n$ the function
\[
 F_2(x)=\sqrt{x_1^2+x_2^2}
\]
is positive, non-linear, homogeneous of degree one, and depends effectively on only two inputs. Its Hessian has rank one, so its graph is flat. In contrast,
\[
 F_n(x)=\sqrt{x_1^2+\cdots+x_n^2}
\]
has
\[
 D^2F_n=\frac1{\|x\|}I-\frac1{\|x\|^3}xx^T,
\]
whose kernel is spanned by $x$ and whose restriction to $x^\perp$ is multiplication by $1/\|x\|$. Hence
\[
 \operatorname{rank}D^2F_n=n-1>1
\]
when $n>2$, and its graph is not flat. Thus, degree-one homogeneity alone does not imply flatness in
dimensions $n>2$.
\end{remark}

\begin{remark}
The degree-one converse in Theorem~1.1 of \cite{CV2013}, as stated, does not extend to dimensions \(n>2\) and requires a correction. More precisely, degree-one homogeneity implies
flatness automatically for two variables, but not in higher
dimensions, as the preceding example shows. Consequently, the
higher-dimensional degree-one branch of Theorem~1.1 in \cite{CV2013}, and
the corresponding constant-returns-to-scale branch of Theorem~4.2 therein,
should be replaced by the degree-one statement in
Theorem~\ref{thm:corrected-homogeneous-classification}. The non-degree-one
branch of Theorem~1.1 is correct and is retained in part~(1) of the present theorem.
Moreover, Corollary~3.1 and Theorem~4.1 of \cite{CV2013}, which concern the
two-input case, are unaffected by this correction.
\end{remark}

The preceding homogeneous theorem can now be combined with
Theorem~\ref{thm:main-rigidity} to obtain the classification
relevant to the variable-exponent production functions studied in this
paper.

\begin{corollary}\label{cor:local-classification-production}
Let
\[
Y(x)=A\prod_{\alpha=1}^{n}x_\alpha^{f_\alpha(x)},
\qquad n>2,
\]
be a nonconstant variable-returns-to-scale Cobb--Douglas production
function on a connected open set \(U\subset\mathbb R_+^n\). If its graph
hypersurface is flat, then its local return-to-scale index is
a nonzero constant \(\beta\), and the following alternatives hold.

\begin{enumerate}
\item If \(\beta\ne1\), then there exist a nonzero constant vector
\(c\in\mathbb R^n\) and a constant \(C>0\), with \(c\cdot x>0\) on
\(U\), such that
\[
Y(x)=C(c\cdot x)^\beta.
\]

\item If \(\beta=1\), set
\[
U_1=\{x\in U:\operatorname{rank}D^2Y(x)=1\}.
\]
For every \(p\in U_1\) there is a neighborhood \(V\subset U\) of
\(p\), an interval \(I\), a \(C^1\) regular curve
\(a:I\to\mathbb R^n\), and a \(C^1\) function \(\tau:V\to I\) such
that
\[
Y(x)=a(\tau(x))\cdot x,
\qquad
a'(\tau(x))\cdot x=0.
\]
On every connected open subset on which \(D^2Y=0\), the production
function \(Y\) is linear.
\end{enumerate}

Conversely, within the class of production functions considered in this
paper, every function in the first alternative has a flat
graph. In the degree-one case, whenever $Y$ admits locally the
representation displayed in part~(2), its graph is flat on
the corresponding neighborhood; the same holds on every linear region.
\end{corollary}

\begin{proof}
By Theorem~\ref{thm:main-rigidity} and
Proposition~\ref{prop:homogeneous-characterization}, the function \(Y\) is
homogeneous of some constant degree \(\beta\). Since \(Y>0\) and is
nonconstant, the degree-zero case is excluded by the argument preceding
Theorem~\ref{thm:corrected-homogeneous-classification}. The assertions now
follow directly from Theorem~\ref{thm:corrected-homogeneous-classification}.
\end{proof}

\begin{remark}
For two inputs,  flatness is equivalent to vanishing Gaussian curvature. In higher dimensions, flatness and vanishing Gauss--Kronecker curvature are different conditions. The present paper concerns flatness, characterized by the vanishing of the Riemann curvature tensor, equivalently by $\operatorname{rank}D^2F\le1$ for graph hypersurfaces. The weaker problem based only on vanishing Gauss--Kronecker curvature remains a separate question.
\end{remark}

\section*{Acknowledgements}
The third author was supported by Ongoing Research
Funding Program (ORF -- 2026 - 413), King Saud University, Riyadh, Saudi
Arabia. The fourth author was supported by Petroleum-Gas University of Ploiești through the Internal Scientific Research Grant No. 14239/22.06.2026.

\section*{Declarations}

\begin{itemize}
\item \textbf{Competing interests} The authors declare no competing interests.
\item \textbf{Data availability} No datasets were generated or analysed during the current study.
\item \textbf{Author contribution} All authors contributed to the research, writing and reviewing of the present
manuscript.
\end{itemize}

\end{document}